\documentclass[12pt,a4paper]{amsart}
\usepackage{amssymb,amsmath,fullpage}
\usepackage[english]{babel}

\usepackage{array}
\usepackage{ifthen}
\usepackage{url}
\usepackage{hyperref}
\usepackage{booktabs}

\newtheorem{theorem}{Theorem}
\newtheorem{lemma}[theorem]{Lemma}

\newtheorem{conjecture}[theorem]{Conjecture}

\theoremstyle{definition}

\numberwithin{equation}{section}
\numberwithin{theorem}{section}

\def\ZZ{\mathbb Z}
\def\RR{\mathbb R}

\renewcommand{\SS}{\mathfrak{S}}

\def\hh{{\rm h}\kern.4pt}

\date{\today}

\begin{document}
\title{On a recursively defined sequence involving the prime counting function}

\author[A. Alkan]{Altug Alkan}

\address{Altug Alkan \newline
         \indent Graduate School of Science and Engineering, \newline
        \indent Piri Reis University,\newline
         \indent Istanbul, Turkey}
         
\email{altug.alkan@pru.edu.tr}      

\author[A. R. Booker]{Andrew R. Booker}
\address{Andrew R. Booker \newline
\indent School of Mathematics, University of Bristol,\newline
\indent Woodland Road, Bristol, BS8 1UG,\newline
\indent United Kingdom}
\email{andrew.booker@bristol.ac.uk}

\author[F. Luca]{Florian Luca}

\address{Florian Luca \newline
         \indent School of Mathematics, University of the Witwatersrand, \newline
        \indent Private Bag X3, Wits 2050,\newline
         \indent Johannesburg, South Africa}
         
\address{Research Group in Algebraic Structures and Applications,\newline 
\indent King Abdulaziz University,\newline 
         \indent Jeddah, Saudi Arabia}
         
\address{Centro de Ciencias Matem\'aticas, UNAM,\newline
    \indent Morelia, Mexico}                  

\email{Florian.Luca@wits.ac.za}

\date{\today}

\pagenumbering{arabic}

\begin{abstract}
We prove some properties of the sequence $\{a_n\}_{n\ge1}$ defined by
$$a_n=\pi(n)-\pi\bigl(\textstyle\sum_{k=1}^{n-1}a_k\bigr).$$
In particular we show that it assumes every non-negative integral value
infinitely often.
\end{abstract}

\maketitle

\section{Introduction}
Let $\pi(x)=\#\{p\text{ prime}:p\le x\}$ denote the prime counting
function. In this paper we consider the sequence $\{a_n\}_{n\ge 1}$ defined by
$$
a_n=\pi(n)-\pi\!\left(\textstyle\sum_{k=1}^{n-1}a_k\right)
\quad\text{for }n\ge 1.
$$
(Here we adopt the convention that the empty sum is $0$, so
$a_1=\pi(1)-\pi(0)=0$.)
This is sequence A335294 in the OEIS \cite{oeis}, and
its initial terms are 
\begin{align*}
&0, 1, 2, 0, 1, 1, 1, 1, 0, 0, 1, 1, 2, 1, 1, 0, 1, 1, 2, 1, 1, 0, 1, 1, 1, 1, 0, 0, 1, 1, 2, 2, 1,\\
&1, 0, 0, 1, 1, 1, 1, 2, 1, 2, 2, 1, 0, 1, 1, 1, 1, 0, 0, 1, 1, 1, 1, 1, 1, 1, 1, 2, 2, 1, 1, 0, 0,\\
&1, 1, 1, 1, 2, 1, 2, 2, 1, 0, 0, 0, 1, 1, 1, 1, 2, 1, 1, 1, 1, 0, 1, 1, 1, 1, 1, 1, 0, 0, 1, \ldots
\end{align*}
At first glance, the sequence is not monotonic and displays a remarkably
slow rate of growth. In this direction, see Table~\ref{tab:Tab1}, which
shows the smallest solutions to $a_n=k$ for each $k\le 13$.
(This is a subsequence of the prime numbers for $k\ge1$;
note that for $k\in\{3,4,5,6,9,10,13\}$ the corresponding $n$ is also
the larger of a twin prime pair.)

\begin{table}[h]
\scriptsize
\begin{tabular}{c|ccccccc}
$k$&$0$&$1$&$2$&$3$&$4$&$5$&$6$\\
$n$&$1$&$2$&$3$&$229$&$3259$&$15739$&$449569$\\ \hline\hline
$k$&$7$&$8$&$9$&$10$&$11$&$12$&$13$\\
$n$&$6958841$&$130259903$&
$2404517671$&$56014949761$&$538155413969$&
$21692297487587$&$21692297487589$
\end{tabular}
\\
\caption{Smallest $n$ satisfying $a_n=k$ \cite[A335337]{oeis}}
\label{tab:Tab1}
\end{table}

Let $s_n$ denote the summatory sequence,
$$
s_n=\sum_{k=1}^na_k\quad\text{for }n\ge0.
$$
Our main result establishes some distributional properties of
$\{a_n\}_{n\ge 1}$ and $\{s_n\}_{n\ge0}$. In order to state them, we
define $g(x)$ to be the maximum distance between a number
$y\le x$ and the largest prime $p\le y$, i.e.\
$$
g(x)=\sup_{y\in[2,x]}\min\{y-p:p\le y\}
\quad\text{for }x\ge2.
$$
Note that $g$ is a continuous, piecewise linear, non-decreasing function,
and
$$
\pi(n)-\pi(n-g(n)-1)\ge1\quad\text{for all integers }n\ge2.
$$
Conjecturally one has $g(x)=O(\log^2x)$;
the best result to date, due to Baker, Harman and Pintz \cite{BHP},
is that $g(x)\le x^{21/40}$ for all sufficiently large $x$.

\begin{theorem}
\label{theorem:1}
The following conclusions hold:
\begin{itemize}
\item[(i)] $a_n\ge 0$ and $a_n-\max\{1,2\pi(a_n)\}\le a_{n+1}\le a_n+1$ for all $n\ge 1$;
\item[(ii)] $a_n=O\bigl(\sqrt{g(n)/\log g(n)}\bigr)$ for all
$n\ge5$;
\item[(iii)] for each $k\ge 0$, there are infinitely many $n$
such that $a_n=k$;
\item[(iv)] $n-g(n)\le s_n\le n-2$ for all $n\ge9$;
\item[(v)] $s_n<n-\frac12g(n)$ for infinitely many $n$.
\end{itemize}
\end{theorem}

\begin{proof}
We begin with the upper estimate in (iv).
Suppose that $s_n\le n$ holds for some $n\ge0$; note that
this is the case for $n=0$. By definition we have
\begin{equation}\label{eq:sn1}
s_{n+1}=s_n+a_{n+1}=s_n+\pi(n+1)-\pi(s_n),
\end{equation}
so that
\begin{equation}\label{eq:snrecurrence}
n+1-s_{n+1}=\bigl(n+1-s_n\bigr)-\bigl(\pi(n+1)-\pi(s_n)\bigr).
\end{equation}
The right-hand side counts the number of non-prime integers in the
interval $(s_n,n+1]$. Since this is non-negative, we have
$s_{n+1}\le n+1$. By induction it follows that $s_n\le n$ for all $n\ge0$.

Next we improve this to $s_n\le n-2$. Suppose $n\ge9$ is
such that $s_{n+i}\le n+i-2$ for $i\in\{0,1,2,3\}$; we verify this directly
for $n=9$. If $s_{n+4}\ge n+3$ then we have
\begin{align*}
n+4-s_{n+4}\le1&\implies(s_{n+3},n+4]
\text{ contains at most one composite number}\\
&\implies n+2\text{ and }n+4\text{ are prime}\\
&\implies n+3,\,n+1,\,n\text{ and }n-1\text{ are composite}\\
&\implies s_{n+3}\ge n+1\implies n+3-s_{n+3}\le 2\\
&\implies(s_{n+2},n+3]
\text{ contains at most two composite numbers}\\
&\implies s_{n+2}\ge n\implies n+2-s_{n+2}\le 2\\
&\implies(s_{n+1},n+2]
\text{ contains at most two composite numbers}\\
&\implies s_{n+1}\ge n-1\implies n+1-s_{n+1}\le 2\\
&\implies(s_n,n+1]
\text{ contains at most two composite numbers}\\
&\implies s_n\ge n-1.
\end{align*}
This contradicts the assumption that $s_n\le n-2$, so we must have
$s_{n+4}\le n+2$. By induction it follows that $s_n\le n-2$ for all
$n\ge9$.

Next, for all $n\ge1$ we have
\begin{equation}\label{eq:anpos}
a_n=\pi(n)-\pi(s_{n-1})\ge\pi(n)-\pi(n-1)\ge0.
\end{equation}
It follows that $s_n$ is non-decreasing, and thus
\begin{equation}\label{eq:andiff}
a_{n+1}-a_n=\bigl(\pi(n+1)-\pi(n)\bigr)-\bigl(\pi(s_n)-\pi(s_{n-1})\bigr)
\le\pi(n+1)-\pi(n)\le1.
\end{equation}
Moreover, by \cite[Corollary~2]{MV}, we have
$$
a_{n+1}\ge a_n-\bigl(\pi(s_{n-1}+a_n)-\pi(s_{n-1})\bigr)
\ge a_n-\max\{1,2\pi(a_n)\}.
$$
This proves (i).\footnote{We note that the lower estimate can be improved
to $a_{n+1}\ge a_n-\pi(a_n)$ for $n\ge4$ and $2\le a_n\le 1731$,
by \cite{GR}.}

Let $n$ be a natural number satisfying
\begin{equation}\label{eq:snlower}
s_n\le n-g(n).
\end{equation}
Note that this holds for $n=3$.
If $s_n\ge n+1-g(n+1)$ then
$$
s_{n+1}=s_n+\pi(n+1)-\pi(s_n)\ge s_n+\pi(n+1)-\pi(n)
\ge s_n\ge n+1-g(n+1).
$$
Otherwise we have $n-g(n)\le s_n<n+1-g(n+1)$, so that
$$
s_{n+1}=s_n+\pi(n+1)-\pi(s_n)\ge n-g(n)+\pi(n+1)-\pi(n-g(n+1)).
$$
Again by the definition of $g$ we have
$\pi(n+1)-\pi(n-g(n+1))\ge1$, so
$$
s_{n+1}\ge n+1-g(n)\ge n+1-g(n+1).
$$
Thus, in either case, \eqref{eq:snlower} holds with $n$ replaced by
$n+1$. By induction, \eqref{eq:snlower} holds for all $n\ge3$,
and this completes the proof of (iv).

Turning to (ii), let $n$ be a natural number, and suppose that
$k=a_n\ge3$. Applying \eqref{eq:andiff} inductively, we see that
\begin{equation}\label{eq:ani}
a_{n-i}\ge k-\bigl(\pi(n)-\pi(n-i)\bigr)
\quad\text{for all }i<n.
\end{equation}
Let $h\ge2$ be the largest integer such that $\pi(h)\le k/3$.
Taking $i=n-1$ in \eqref{eq:ani} we see that
$\pi(n)\ge k\ge 3\pi(h)>\pi(h)$, whence $h<n$.
Moreover, by the prime number theorem we have
$h\asymp k\log{k}$.
By \cite[Corollary~2]{MV}, for any non-negative integer $i\le h$ we have
$$
\pi(n)-\pi(n-i)\le\max\{1,2\pi(i)\}\le2\pi(h)\le 2k/3,
$$
so that $a_{n-i}\ge k/3$.
Therefore,
$$
s_n-s_{n-h}=\sum_{i=0}^{h-1}a_{n-i}\ge\frac{hk}{3}\gg k^2\log{k}.
$$
By (iv), $s_n-s_{n-h}=h+O(g(n))\ll k\log{k}+g(n)$.
Thus, $k^2\log{k}\ll k\log{k}+g(n)$, and (ii) follows.

Next, set $h_n=n-s_n$. Recall from \eqref{eq:snrecurrence}
that $h_n$ is the number of non-prime integers in the interval
$(s_{n-1},n]$.  Let $p,q$ be a pair of consecutive odd primes, and
set $n=(p+q)/2$.  If $s_{n-1}<p$ then $h_{n-1}\ge n-p=(q-p)/2$.
Otherwise, the interval $(s_{n-1},n]$ contains no primes, so
$h_n=n-s_{n-1}=h_{n-1}+1$ and $s_n=s_{n-1}$; repeating this
argument with $n$ replaced by $n+i$, it follows by induction that
$$
h_{n+i}=h_{n-1}+i+1
\quad\text{for }0\le i<q-n=(q-p)/2.
$$
In particular, $h_{q-1}=h_{n-1}+(q-p)/2\ge(q-p)/2$.
Hence, in any case we find that
$$
\max_{p\le m<q}h_m\ge(q-p)/2.
$$
Choosing $p$ and $p$ attaining a maximal gap, we have $q-p=g(q-1)+1$,
and (v) follows.\footnote{By \cite{FGKM}, it also follows that
$n-s_n\gg\frac{\log{n}\log\log{n}\log\log\log\log{n}}{\log\log\log{n}}$
infinitely often.}

Next we prove (iii). First note that if there were only finitely many
$n$ with $a_n=0$ then we would have $s_n\ge n-O(1)$, contradicting (v);
hence (iii) is true for $k=0$.  Since $a_n$ can increase by at most $1$ at
each step and there are infinitely many $n$ with $a_n=0$, to complete the
proof of (iii) it suffices to show that $\{a_n\}_{n\ge1}$ is unbounded.

To that end, for a given integer $m\ge2$ we apply the main result of
\cite{BFB} to find a sequence of $m+1$ consecutive primes with a large
gap followed by a relatively dense cluster.  Precisely, let $k=k_{m+1}$
in the notation of \cite{BFB}, and set $b_j=-m^{k+1-j}$ for $j=1,\ldots,k$.
Then it is easy to see that the polynomial $\prod_{j=1}^k(x+b_j)$
has no fixed prime divisor, so by \cite[Theorem~1]{BFB} there exists
a subset $\{h_0,\ldots,h_m\}\subseteq\{b_1,\ldots,b_k\}$ such that
$x+h_0,\ldots,x+h_m$ are consecutive primes for infinitely many $x\in\ZZ$.

Fix any such $x$, denote the corresponding primes by $p_0,\ldots,p_m$, and
write $h_i=-c^{k+1-j_i}$, where $1\le j_0<\cdots<j_m\le k$.
Then
\begin{align*}
(m-1)(h_m-h_1)&=(m-1)(m^{k+1-j_1}-m^{k+1-j_m})\\
&<m^{k+2-j_1}-m^{k+1-j_1}
\le m^{k+1-j_0}-m^{k+1-j_1}=h_1-h_0,
\end{align*}
so that
$$
p_1-p_0>(m-1)(p_m-p_1)\ge(m-1)p_m-(p_1+p_2+\cdots+p_{m-1}).
$$

Next, define sequences $\{s_n'\}_{n\ge p_1}$,
$\{a_n'\}_{n>p_1}$ and $\{d_n\}_{n\ge p_1}$ by
\begin{equation}\label{eq:snprime}
s_{p_1}'=p_1,\qquad
s_{n+1}'=s_n'+\pi(n+1)-\pi(s_n')\quad\text{for }n\ge p_1,
\end{equation}
$$
a_n'=s_n'-s_{n-1}'\quad\text{for }n>p_1
$$
and
$$
d_n=s_n'-s_n\quad\text{for }n\ge p_1.
$$
By the same proof as for $s_n$, we see that $s_n'\le n$ and $a_n'\ge0$
for all $n>p_1$. Further, subtracting \eqref{eq:snprime} and
\eqref{eq:sn1}, we find that
$$
d_{n+1}=d_n-\bigl(\pi(s_n+d_n)-\pi(s_n)\bigr)
\quad\text{for }n\ge p_1.
$$
It follows that $0\le d_{n+1}\le d_n$, so that
$$
a_n'=a_n+d_n-d_{n-1}\le a_n\quad\text{for }n>p_1.
$$

A straightforward inductive argument now shows that
\begin{align*}
s_n'&=p_0&&\text{for }p_0\le n<p_1,\\
s_n'&=p_0+n-p_1+1&&\text{for }p_1\le n<p_2,\\
s_n'&=p_0+(p_2-p_1)+2(n-p_2+1)&&\text{for }p_2\le n<p_3,\\
&\;\;\vdots\\
s_n'&=p_0+(p_2-p_1)+2(p_3-p_2)+\cdots\\
&\quad+(m-2)(p_{m-1}-p_{m-2})+(m-1)(n-p_{m-1}+1)
&&\text{for }p_{m-1}\le n<p_m.
\end{align*}
In particular,
$$
s_{p_m-1}'=p_0+(m-1)p_m-(p_1+p_2+\cdots+p_{m-1})<p_1,
$$
so that
$$
a_{p_m}\ge a_{p_m}'=\pi(p_m)-\pi(s_{p_m-1}')=m.
$$
Since $m$ was arbitrary, this completes the proof of (iii).
\end{proof}

\section{Some conjectures}
It follows from (ii) and (iv) that
\begin{equation}\label{eq:limits}
\lim_{n\to \infty}\frac{a_n}{n}=0
\quad\text{and}\quad
\lim_{n\to\infty}\frac{s_n}{n}=1.
\end{equation}
We further conjecture the following.
\begin{conjecture}\
\label{conj:1}
\begin{itemize}
\item[(A)]For any $k\ge0$, the set $\{n\ge1:a_n=k\}$ has a positive density
$\delta_k$, satisfying
$\delta_1>\delta_0>\delta_2>\delta_3>\delta_4>\ldots$
\item[(B)]$\liminf_{n\to\infty}(n-s_n)<\infty$.
\item[(C)]For any integer $b\ge2$, the number
$A(b)=\sum_{n\ge1}a_nb^{-n}\in\RR$ is transcendental.
\end{itemize}
\end{conjecture}
In connection with (B),
it seems likely from numerical computations that $s_n=n-2$ infinitely
often; by (iv) this would imply that $\liminf_{n\to\infty}(n-s_n)=2$.
If $\{a_n\}_{n\ge1}$ were an automatic sequence then (C) would follow
from the main result in \cite{ABL}.

\begin{table*}[!h]
\scriptsize
\begin{center}
\begin{tabular}{ccccccc}
\noalign{\smallskip}
 $ $ & $ $ & $ $ & $k$ & $ $ & $ $ & $ $\\
\noalign{\smallskip}\cmidrule{2-7}\noalign{\smallskip}
\noalign{\smallskip}
$ i $ & $0$ & $1$ & $2$ & $3$ & $4$ & $5$ \\
\noalign{\smallskip}\hline\noalign{\smallskip}
$1 $ & $4$ & $5$ & $1$ & $0$ & $0$ & $0$\\
$ 2 $ & $21$ & $65$ & $14$ & $0$ & $0$ & $0$\\
$ 3 $ & $219$ & $577$ & $195$ & $9$ & $0$ & $0$\\
$ 4 $ & $2663$ & $4990$ & $2065$ & $275$ & $7$ & $0$\\
$ 5 $ & $27671$ & $48507$ & $20265$ & $3287$ & $257$ & $13$ \\
$ 6 $ & $284408$ & $475421$ & $199765$ & $36779$ & $3443$ & $181$\\
$ 7 $ & $2918543$ & $4650175$ & $1991476$ & $395418$ & $41464$ & $2800$ \\
$ 8 $ & $29607905$ & $45960839$ & $19809319$ & $4108991$ & $473258$ & $37723$ \\
$ 9 $ & $299530722$ & $455176760$ & $197289962$ & $42282008$ & $5235205$ & $456865$\\
$ 10 $ & $3022594978$ & $4517557589$ & $1965289965$ & $432413509$ & $56484650$ & $5291355$ \\
$ 11 $ & $30450733004$ & $44894741076$ & $19590459294$ & $4400511075$ & $599692839$ & $59396517$\\
$ 12 $ & $306392386246$ & $446604857931$ & $195374867235$ & $44626996156$ & $6295691446$ & $652786704$ \\
$ 13 $ & $3080065196771$  & $4446030725007$ & $1949223822125$ & $451486351994$ & $65543491929$ & $7053078276$ \\
$ 14 $ & $30940285500711$ & $44287714979733$ & $19452797930000$ &
$4559198048883$ & $678055064108$ & $75277782875$ \\
\hline
\noalign{\smallskip}
\end{tabular}
\caption {Values of $\#\{n\le 10^{i}: a_n=k\}$ for  $0 \le k \le 5$ and $1 \le i \le 14$.}
\label{tab:Tab2}
\end{center}
\end{table*}

These conjectures are supported by numerical evidence, such as in
Table~\ref{tab:Tab2}. We provide the following theoretical evidence.
\begin{theorem}\
\label{theorem:2}
\begin{itemize}
\item[(i)]$\#\{n\ge1:a_n=k\}$ has positive lower density for at least
one $k\in\{0,1\}$.
\item[(ii)]We have
$\#\{n\le x:a_n=0\}\gg\frac{x\log\log{x}}{\log^2x}$
for $x\ge3$ under the hypothesis that $g(x)\ll(\log{x})^C$ for some $C>1$,
and
$\#\{n\le x:a_n=0\}\ge\exp\bigl((\log{x})^{\frac14-o(1)}\bigr)$
unconditionally.
\item[(iii)]$\liminf_{n\to\infty}\frac{n-s_n}{\log{n}}\le1$.
\item[(iv)]The number $A(b)$ is irrational.
\end{itemize}
\end{theorem}

\subsection*{Proof of (i) and (ii)}
We begin by setting some notation to be used in the proof.
Let $x>0$ be a large real number, and let $r,T\in\ZZ$ be parameters,
to be specified in due course, satisfying
$$
2\le r\le T\le r(\log{x})^{\frac{r-1}{4r+2}}/\log\log{x}.
$$
We regard $r$ as fixed throughout the proof, so the meaning of $\ll$,
$O$, $o$, ``sufficiently large'', etc.\ may depend implicitly on $r$.
Let $K\ge1$ be a large (absolute) constant, and define
$$
H_j=Kj^2(\log{x})(\log{T})
\quad\text{for }0\le j\le T.
$$

Next, set $N=\lfloor x\rfloor$ and
$$
N_k=\#\{1\le n\le x:a_n=k\}
\quad\text{for }k\ge0.
$$
Then
$$
N_0+N_1+N_2+\cdots=N
\quad\text{and}\quad
N_1+2N_2+\cdots=s_N,
$$
so that
$$
N_0-(N_2+2N_3+\cdots)=N-s_N>0.
$$
Thus
$$
N_0+(N_2+N_3+\cdots)\le
N_0+(N_2+2N_3+\cdots)<2N_0.
$$
This also shows that $N_1+2N_0>N$, so that
$\max\{N_0,N_1\}>\frac13N$. It follows that at least one of the sets
$\{n\ge1:a_n=k\}$ for $k\in\{0,1\}$ has lower density $\ge\frac13$.
This proves (i).

Next, setting $J=\{1\le n\le x:a_n\ne1\}$, we have
$\#J<2N_0$.
Let
$$
L=\{\ell\in\ZZ:1\le\ell\le x\text{ and }
a_n\ne1\text{ for some }n\in[\ell-1,\ell+H_T]\cap\ZZ_{>0}\}.
$$
By \cite[Corollary~5]{MV}, the
number of primes contained in $L$ is at most
$$
2\pi(H_T+2)(\#J+1)\le4\pi(H_T+2)N_0.
$$
Suppose, for the sake of contradiction, that
$4\pi(H_T+2)N_0\le\frac12\pi(x)$.
From now on we consider primes $p\in[1,x]\setminus L$, which is at least
half of the primes $p\le x$. These primes have the property that
$a_n=1$ for all integers $n$ satisfying $p-1\le n\le p+H_T$.

We need two easy facts about primes.
\begin{lemma}\label{lem:1}
Let $p_i$ denote the $i$th prime.
For a suitable choice of the constant $K$ and all sufficiently large
$x$, there are at most $\frac14\pi(x)$ primes
$p_i\le x$ for which there exists $j\in\{1,\ldots,T\}$ satisfying
$p_{i+j}-p_i>H_j$.
\end{lemma}
\begin{proof}
Fix $j\in\{1,\ldots,T\}$. Then, since $j\le T=o(\pi(x))$, we have
$$
\sum_{i=1}^{\pi(x)}(p_{i+j}-p_i)<\sum_{k=\pi(x)+1}^{\pi(x)+j}p_k
=(1+o(1))jx,
$$
by the prime number theorem.
Thus, the number of $i$ such that $p_{i+j}-p_i>H_j$ is
$\ll jx/H_j=x/(jH_1)$. Summing this over all $j\le T$, we get a bound of 
$$
\ll\frac{x}{H_1}\sum_{j\le T}\frac1j\ll\frac{x\log T}{H_1}
\ll\frac{\pi(x)}{K}.
$$
For $K$ sufficiently large this is less than $\frac14\pi(x)$.
\end{proof}

\begin{lemma}\label{lem:2}
For a fixed choice of $r\ge2$, there are
at most $o(\pi(x))$ primes $p_i\le x$ satisfying the following
conditions:
\begin{itemize}
\item[(i)]$p_{i+j}-p_i\le H_j$ for all $j\in\{0,\ldots,T\}$;
\item[(ii)]there are vectors
$(j_1,\ldots,j_r),(j_1',\ldots,j_r')\in\ZZ^r$
such that
$$
0\le j_1<j_2<\cdots<j_r\le T,
\quad
0\le j_1'<j_2'<\cdots<j_r'\le T
$$
and
$$
p_{i+j_1}-p_{i+j_1'}=\cdots=p_{i+j_r}-p_{i+j_r'}\ne0.
$$
\end{itemize}
\end{lemma}
\begin{proof}
Let $p_i$ be such a prime, and set
$$
h_j=p_{i+j}-p_i\quad\text{for }j\in\{0,\ldots,T\}.
$$
Let $(j_1,\ldots,j_r),(j_1',\ldots,j_r')$ be as in (ii), and write
$$
\{j_1,\ldots,j_r\}\cup\{j_1',\ldots,j_r'\}=\{\ell_1,\ldots,\ell_k\},
$$
with $\ell_1<\cdots<\ell_k$.
From our hypotheses it is clear that $r+1\le k\le 2r$.
Let
$$
d=h_{j_1}-h_{j_1'}=\cdots=h_{j_r}-h_{j_r'}
$$
denote the common difference. Swapping
$(j_1,\ldots,j_r)$ and $(j_1',\ldots,j_r')$ if necessary, we may assume
without loss of generality that $d>0$, and it follows that
$j_s>j_s'$ for each $s\in\{1,\ldots,r\}$.

For a fixed value of $k$, there are
$O(T^k)$ ways of choosing $\{j_1,\ldots,j_r\}$ and
$\{j_1',\ldots,j_r'\}$ of total cardinality $k$.
If $k=2r$ then for each choice of indices, there are at most
$H_T^{r+1}$ choices for the pair of vectors
$v=(h_{j_1},\ldots,h_{j_r})$, $v'=(h_{j_1'},\ldots,h_{j_r'})$,
since $v'$ is determined by $v$ and $d$.
If $k<2r$ then there are $2r-k$ pairs $(s,t)$ such that
$j_s=j_t'$; for each pair we have
$d=h_{j_t}-h_{j_t'}=h_{j_t}-h_{j_s}$, so that $h_{j_t}$ is determined by
$h_{j_s}$ and $d$.
Hence, in general there are at most $H_T^{k+1-r}$ choices for $v,v'$
for a given choice of indices.
Thus, in total we find
$$
\ll T^kH_T^{k+1-r}\ll T^{3k+2-2r}((\log{T})(\log{x}))^{k+1-r}
$$
choices for $v,v'$ for our fixed $k$.

Let us first suppose that $\ell_1>0$. Then
$n=p_i$ is an integer such that 
the $k+1$ distinct linear forms $n,n+h_{\ell_1},\ldots,n+h_{\ell_k}$
are all prime. By \cite[Ch.~II, Satz~4.2]{Prachar},
the number of such $n\le x$ is 
$$
\ll_k\frac{x}{(\log x)^{k+1}}\left(\frac{E}{\varphi(E)}\right)^k,
\quad\text{where }
E=\prod_{1\le s\le k}h_{\ell_s}
\cdot\prod_{1\le s<t\le k}\bigl(h_{\ell_t}-h_{\ell_s}\bigr).
$$
Since $h_{\ell_1},\ldots,h_{\ell_k}\le H_T\ll\log^2x$, we have
$\frac{E}{\varphi(E)}\ll\log\log\log{x}$.
Hence, the number of possibilities for $p_i$ is
\begin{align*}
&\ll\frac{T^{3k+2-2r}(\log{T})^{k+1-r}x(\log\log\log{x})^k}{(\log{x})^r}
\le\frac{T^{4r+2}(\log{T})^{r+1}x(\log\log\log{x})^{2r}}{(\log{x})^r}\\
&\ll\frac{x(\log\log\log{x})^{2r}}{(\log{x})(\log\log{x})^{3r+1}}
=o(\pi(x)).
\end{align*}

If $\ell_1=0$ then we lose one linear form, but gain from the fact that
$j_1'$ and $h_{j_1'}$ are fixed at $0$. This effectively replaces $k$
by $k-1$ in the above analysis, so we again find $o(\pi(x))$ possibilities
for $p_i$. Finally, summing over $k\in\{r+1,\ldots,2r\}$ concludes the proof
of the lemma.
\end{proof}

The following is Lemma~5.1 in \cite{FKL}. 
\begin{lemma}\label{sieve}
There is a positive constant $\delta$ so that the following holds.
Let $a_1,\ldots,a_k$ be positive integers, let $b_1,\ldots,b_k$ be
integers and let $\xi(p)$ be the number of solutions of
$\prod_{i=1}^k (a_i n+b_i) \equiv 0\pmod{p}$.
If $x\ge 10$, $\displaystyle{1\le k\le\delta\frac{\log x}{\log\log x}}$ and
$$
B:=\sum_p\left(\frac{k-\xi(p)}{p}\right)\log p\le\delta\log x,
$$
then the number of integers $n\le x$ for which
$a_1n+b_1,\ldots,a_kn+b_k$ are all prime and $>k$ is
\begin{equation}
\label{eq:TTT}
\ll\frac{2^kk!\SS{x}}{(\log x)^k}
\exp\!\left(O\!\left(\frac{kB+k^2\log\log{x}}{\log{x}}\right)\right),
\quad\text{where }
\SS=\prod_p\left(1-\frac{\xi(p)}{p}\right)\left(1-\frac1p\right)^{-k}.
\end{equation}
\end{lemma}

We are now ready to go. As we said, we work with primes $p_i\le x$
that are not in $L$. The number of them is at least $\frac12\pi(x)$.
We discard all $p_i$ such that $p_{i+j}-p_i>H_j$ holds for some
$j=1,\ldots,T$. By Lemma~\ref{lem:1}, there are at most $\frac14\pi(x)$
such primes. Next, applying Lemma~\ref{lem:2}, by removing a further
$o(\pi(x))$ values of $p_i$, we may assume that as $j$ and $j'$ range
over $\{0,\ldots,T\}$, each non-zero difference $p_{i+j}-p_{i+j'}$
occurs with multiplicity at most $r-1$.  After this we are left with at
least $(\frac14-o(1))\pi(x)$ primes $p_i$.

Set $c=p_i-s_{p_i}$. By Theorem~\ref{theorem:1}(iv) and the definition of
$N_0$, we have
$$
0<c\le M:=\min\{g(x),N_0\}.
$$
Now consider $p_i,p_{i+1},\ldots,p_{i+T}$.
These are of the form $p_{i+j}=p_i+h_j$ for some $h_j\le H_j$, as in the proof
of Lemma~\ref{lem:2}.
On the other hand, $p_{i+T}-p_i\le H_T$, and since
$a_n=1$ for $p_i-1\le n\le p_i+H_T$,
we have $s_n=n-c$ for $p_i-2\le n\le p_{i+T}$.
Applying \eqref{eq:andiff} with $n=p_{i+j}-1$, we have
\begin{align*}
0&=a_{p_{i+j}}-a_{p_{i+j}-1}=
\bigl(\pi(p_{i+j})-\pi(p_{i+j}-1)\bigr)
-\bigl(\pi(s_{p_{i+j}-1})-\pi(s_{p_{i+j}-2})\bigr)\\
&=1-\bigl(\pi(p_{i+j}-1-c)-\pi(p_{i+j}-2-c)\bigr).
\end{align*}
Hence, $p_{i+j}-1-c=p_i+h_j-c-1$ is prime.

Therefore, $n=p_i$ is such that $n+h_j$ and $n+h_j-c-1$ are all primes
for $j=0,\ldots,T$. This is $2T+2$ linear forms, but they might not all
be distinct.
Let $m$ be the cardinality of the intersection
$$
\{h_j:0\le j\le T\}\cap\{h_j-c-1:0\le j\le T\}.
$$
Then there exist $j_0<j_1<\cdots<j_m$ and $j_0'<j_1'<\cdots<j_m'$
with
$$
h_{j_1}-h_{j_1'}=\cdots=h_{j_m}-h_{j_m'}=c+1,
$$
so that
$$
p_{i+j_1}-p_{i+j_1'}=\cdots=p_{i+j_m}-p_{i+j_m'}>0.
$$
By our construction, we must have $m<r$; in particular, there are
at least $2T+3-r$ distinct forms among the $n+h_j$ and $n+h_j-c-1$
for $j=0,\ldots,T$. Hence we may apply Lemma~\ref{sieve} for some
$k\in[2T+3-r,2T+2]\cap\ZZ$.

We need to check the hypothesis on $B$ and estimate some of the parameters
in \eqref{eq:TTT}. For $B$, we partition the primes into
$S_1\cup S_2\cup S_3$, where
$$
S_1=\bigl\{p:p\le\log^2x\text{ or }p\mid(c+1)\bigr\},
\quad S_2=\{p:\xi(p)<k\}\setminus S_1,
\quad S_3=\{p:\xi(p)=k\}\setminus S_1.
$$

Since $c\le x$, $c+1$ has $O(\log{x}/\log\log{x})$ prime factors
exceeding $\log^2x$. Hence,
\begin{align*}
\sum_{p\in S_1}\left(\frac{k-\xi(p)}{p}\right)\log p
&\le k\sum_{p\le\log^2x}\frac{\log p}{p}
+k\sum_{\substack{p\mid(c+1)\\p>\log^2x}}\frac{\log{p}}{p}\\
&\ll T\log\log{x}+\frac{T}{\log{x}}
\ll T\log\log{x}.
\end{align*}

For any prime $p\in S_2$, there is a double solution $n$ modulo $p$ to 
$$
\prod_{0\le j\le T}(n+h_j)
\cdot
\prod_{\substack{0\le j\le T\\h_j-c-1\notin\{h_0,\ldots,h_T\}}}
(n+h_j-c-1)\equiv 0\pmod p.
$$
If the double root comes from the forms $n+h_j$ for $j=0,\ldots,T$,
we get that $p$ divides $h_{j_2}-h_{j_1}$ for some $j_1,j_2$ with
$0\le j_1<j_2\le T$.
But this is impossible since $p>\log^2x$
and $h_j\le H_T<\log^2x$ for large $x$. The same
argument shows that the double solution cannot come from two factors
of the form $n+h_j-c-1$ for $j\in\{0,\ldots,T\}$. So any double root must
appear once from the first set of forms and once from the second, so
that $p$ divides $c+1+h_{j'}-h_j\ne0$ for some $j,j'\in\{0,\ldots,T\}$.
These numbers all lie in the interval $[c+1-H_T,c+1+H_T]$, and since $c\le x$,
each has $O(\log{x}/\log\log{x})$ prime factors exceeding $\log^2x$.
Thus,
$$
\#S_2\ll\frac{H_T\log{x}}{\log\log{x}}
\ll\log^3x.
$$
Moreover, writing $m=k-\xi(p)$, there exist $j_1<\cdots<j_m$,
$j_1'<\cdots<j_m'$ such that
$$
h_{j_1}-h_{j_1'}\equiv\cdots\equiv h_{j_m}-h_{j_m'}\equiv
c+1\pmod{p}.
$$
Since $p\nmid(c+1)$ and $2H_T+1<\log^2x$ for large $x$,
this implies that
$$
h_{j_1}-h_{j_1'}=\cdots=h_{j_m}-h_{j_m'}\ne0.
$$
Thus we have $m<r$, so that $\xi(p)\ge k+1-r$.
Therefore
$$
\sum_{p\in S_2}\left(\frac{k-\xi(p)}{p}\right)\log{p}
\le(r-1)\sum_{\log^2x<p\le O(\log^3x)}\frac{\log p}{p}\ll\log\log{x}.
$$
Finally, the primes in $S_3$ don't contribute to $B$.
Thus, the bound on $B$ holds, and in fact $B=O(T\log\log{x})$.

We now estimate \eqref{eq:TTT}. Since $B=O(T\log\log x)$,
the factor involving $\exp$ tends to $1$ as $x\to\infty$,
so it is smaller than $2$ for large
$x$. In the expression for $\SS$, the primes $p\in S_1$ contribute
at most
\begin{align*}
\left(\frac{c+1}{\varphi(c+1)}\right)^k
\prod_{p\le\log^2x}\left(1-\frac1p\right)^{-k}
&=O(\log\log{x})^k
\exp\Biggl(k\sum_{p\le\log^2x}\frac{O(1)}{p}\Biggr)\\
&=\exp\bigl(O(T\log\log\log x)\bigr).
\end{align*}
The contribution from $p\in S_2$ is at most
\begin{align*}
\prod_{p\in S_2}
\left(1-\frac{k+1-r}{p}\right)&\left(1-\frac1p\right)^{-k}
=\prod_{p\in S_2}\left(1-\frac{k+1-r}{p}\right)
\left(1+\frac{k}{p}+O\!\left(\frac{k^2}{p^2}\right)\right)\\
&=\prod_{p\in S_2}\left(1+O\!\left(\frac1p\right)\right)
=\exp\!\left(\sum_{p\in S_2}\frac{O(1)}{p}\right)=e^{O(1)}.
\end{align*}
Similarly, from $p\in S_3$ we get a contribution of
$$
\prod_{p\in S_3}\left(1-\frac{k}{p}\right)
\left(1-\frac1p\right)^k
=\exp\Biggl(k\sum_{p>\log^2x}\frac{O(1)}{p^2}\Biggr)
=\exp\!\left(O\!\left(\frac{k}{\log{x}}\right)\right)=e^{O(1)}.
$$
Thus, in total we have
$$
\SS=\exp\bigl(O(T\log\log\log x)\bigr).
$$

Applying Lemma~\ref{sieve}, the number of $n\le x$ of this form is
$$
\ll\frac{2^kk!x}{(\log{x})^k}\exp\bigl(O(T\log\log\log{x})\bigr).
$$
Since $2k\le4T+4<\log{x}$ for large $x$, this is largest when $k=2T+3-r$.
Using also that
$$
2^{2T+3-r}(2T+3-r)!=T^{2T+8-4r}e^{O(r\log{T})+O(T)}
=T^{2T+8-4r}e^{O(T)},
$$
we obtain
$$
\ll\frac{T^{2T+8-4r}\pi(x)}{(\log{x})^{2T+2-r}}
\exp\bigl(O(T\log\log\log{x})\bigr).
$$
This is for fixed $c,h_1,\ldots,h_T$. The number of choices for these
parameters is at most
$$
MH_1\cdots H_T=M(T!)^2H_1^T
\le MT^{2T}(\log{x})^T\exp\bigl(O(T\log\log\log x)\bigr).
$$
Thus, in total the number of possibilities is
$$
\ll\frac{M\pi(x)}{\exp\bigl((T+2-r)\log\bigl(\frac{\log{x}}{T^4}\bigr)\bigr)}
\exp\bigl(O(T\log\log\log x)\bigr).
$$

This must account for at least $(\frac14-o(1))\pi(x)$ primes, so
for sufficiently large $x$ we have
$$
M=\min\{g(x),N_0\}\gg\exp\!\left((T+2-r)\log\!\left(\frac{\log{x}}{T^4}\right)
-O(T\log\log\log{x})\right).
$$
If $g(x)\ll(\log{x})^C$ for some $C>1$, then taking $r=2$ and
$T=\lfloor{C}\rfloor+1$
results in a contradiction for sufficiently large $x$.
Hence, our hypothesis that
$2\pi(H_T+2)N_0\le\frac12\pi(x)$ must be false,
and it follows that $N_0\gg x(\log\log{x})/\log^2x$.

On the other hand, assuming that $N_0\ll x/\log^2x$ (and making no
hypothesis on $g(x)$), we can take
$T=\lfloor{r(\log{x})^{\frac{r-1}{4r+2}}/\log\log{x}}\rfloor$, and we conclude
that $N_0\ge\exp\bigl((\log{x})^{\frac{r-1}{4r+2}}\bigr)$ for all sufficiently
large $x$. Since this is true for every $r\ge2$, we have
$N_0\ge\exp\bigl((\log{x})^{\frac14-o(1)}\bigr)$.

\subsection*{Proof of (iii)}
Consider positive integers $M<N$, and let
$h=\min\{n-s_n:M\le n<N\}$. Then
\begin{align*}
s_N-s_M&=\sum_{n=M}^{N-1}a_{n+1}
=\sum_{n=M}^{N-1}\bigl(\pi(n+1)-\pi(s_n)\bigr)
\ge\sum_{n=M}^{N-1}\bigl(\pi(n+1)-\pi(n-h)\bigr)\\
&=\sum_{i=0}^h\bigl(\pi(N-i)-\pi(M-i)\bigr)
\ge(h+1)\bigl(\pi(N-h)-\pi(M)\bigr).
\end{align*}
By Theorem~\ref{theorem:1}(iv) and \cite{BHP}
we have $h\le g(M)\le M^{21/40}$ for sufficiently large $M$.
Choosing $N=\lceil{M+M^{7/12}}\rceil$,
by \cite{Heath-Brown} we have
$$
\pi(N-h)-\pi(M)=(1+o(1))\frac{N-M}{\log{M}}
\quad\text{as }M\to\infty.
$$
On the other hand,
$$
s_N-s_M\le N-(M-g(M))\le(1+o(1))(N-M),
$$
so that $h\le(1+o(1))\log{M}$.
Thus, every sufficiently large interval
$[M,M+M^{7/12})$ contains an integer $n$ with
$n-s_n\le(1+o(1))\log{n}$.

\subsection*{Proof of (iv)}
Let $N$ be a large natural number, and write $\{1,\ldots,N\}$ as a
disjoint union $I_1\cup\cdots\cup I_J$ of intervals $I_j$
such that $a_n$ is constant on each $I_j$ and $J$ is as small as
possible. Setting $m_j=\max I_j$ for $j\le J$, we have either
$m_j=N$ or $a_{m_j+1}\ne a_{m_j}$.
From \eqref{eq:andiff} we see that if $a_{n+1}\ne a_n$ then either
$\pi(n+1)\ne\pi(n)$ or $\pi(s_n)\ne\pi(s_{n-1})$.
Since both sequences $\pi(n)$ and $\pi(s_{n-1})$ are non-decreasing and
$s_n\le n$ for all $n$, it follows that
$$
J\le 1+\#\{n<N:a_{n+1}\ne a_n\}\le 1+2\pi(N).
$$
Thus, for at least one of the intervals, say $I_j=\{n_1,\ldots,n_2\}$,
we have
$$
\#I_j=n_2-n_1+1\ge\frac{N}{1+2\pi(N)}.
$$
By the prime number theorem, for any fixed $\varepsilon>0$ this exceeds
$(\frac12-\varepsilon)\log{N}$ for all sufficiently large $N$.

Suppose $A(b)=u/v$ is rational. Then, multiplying by $v(b-1)b^{n_1-1}$,
we obtain
$$
u(b-1)b^{n_1-1}=v(b-1)b^{n_1-1}\sum_{n=1}^\infty\frac{a_n}{b^n}.
$$
Let $c=a_{n_1}$. Since $a_n$ is constant for $n_1\le n\le n_2$, we have
\begin{align*}
u(b-1)b^{n_1-1}
&=v(b-1)\sum_{n=1}^{n_1-1}a_nb^{n_1-1-n}
+v(b-1)c\sum_{n=n_1}^{n_2}b^{n_1-1-n}
+v(b-1)\sum_{n\ge n_2+1}a_nb^{n_1-1-n}\\
&=v(b-1)\sum_{n=1}^{n_1-1}a_nb^{n_1-1-n}
+vc(1-b^{n_1-n_2-1})
+v(b-1)b^{n_1-n_2-1}\sum_{m\ge1}\frac{a_{n_2+m}}{b^m}.
\end{align*}
Hence,
\begin{equation}\label{eq:int}
vb^{n_1-n_2-1}\left((b-1)\sum_{m\ge1}\frac{a_{n_2+m}}{b^m}-c\right)
=u(b-1)b^{n_1-1}-v(b-1)\sum_{n=1}^{n_1-1}a_nb^{n_1-1-n}-vc
\end{equation}
is an integer.

On the other hand, since $0\le a_{n_2+m}\le c+m$ for
$m\ge1$, we have
$$
-c\le(b-1)\sum_{m\ge1}\frac{a_{n_2+m}}{b^m}-c
\le(b-1)\sum_{m\ge1}\frac{c+m}{b^m}-c=\frac{b}{b-1}.
$$
Hence the left-hand side of \eqref{eq:int} is bounded in modulus by
$$
\frac{(c+2)v}{b^{n_2-n_1+1}}
\le\frac{(c+2)v}{N^{(\frac12-\varepsilon)\log{b}}}
\ll\frac{v\sqrt{g(N)}}{N^{(\frac12-\varepsilon)\log{b}}}.
$$
By \cite{BHP}, we have $g(N)\le N^{21/40}$ for sufficiently large $N$.
Since $\log{b}\ge\log2>\frac{21}{40}$, for small enough $\varepsilon$ this
expression tends to $0$ as $N\to\infty$.
Since it has to be an integer,
it must be $0$ for all sufficiently large $N$.

Therefore,
$$
\sum_{m\ge1}\frac{a_{n_2+m}}{b^m}=\frac{c}{b-1}
=\sum_{m\ge1}\frac{c}{b^m}.
$$
By Theorem~\ref{theorem:1}(iii) there exists $n_3>n_2$ such that $a_{n_3}=0$.
Thus, we have
$$
\sum_{j=1}^\infty\frac{a_{n_3+j}}{b^{n_3+j}}
=\sum_{m\ge1}\frac{c}{b^m}-\sum_{m=1}^{n_3-n_2-1}\frac{a_{n_2+m}}{b^m}
=\sum_{m=1}^{n_3-n_2-1}\frac{c-a_{n_2+m}}{b^m}+\frac{cb^{n_2-n_3+1}}{b-1}.
$$
Multiplying both sides by $(b-1)b^{n_3-1}$, we see that the right-hand
side is an integer, so
$$
(b-1)\sum_{j=1}\frac{a_{n_3+j}}{b^{j+1}}\in\ZZ.
$$
On the other hand, we have $0\le a_{n_3+j}\le j$, and by
Theorem~\ref{theorem:1}(iii) both inequalities are strict for infinitely
many $j$. Hence,
$$
0<(b-1)\sum_{j\ge1}\frac{a_{n_3+j}}{b^{j+1}}
<(b-1)\sum_{j\ge1}\frac{j}{b^{j+1}}=1.
$$
This is a contradiction, so $A(b)$ must be irrational.

\section{Generalizations and suggestions for further work}
The sequence $a_n$ admits a vast generalization via sequences of the form
$$
a_f(n)=\pi(f(n))-\pi\!\left({\textstyle\sum_{k=1}^{n-1}a_f(k)}\right)
$$
for various functions $f$. For instance, choosing $f(n)=tn$ for a fixed
integer $t>0$, our proof of \eqref{eq:limits} can be generalized to show that
$$
\lim_{n\to\infty}\frac{a_f(n)}{n}=0
\quad\text{and}\quad
\lim_{n\to\infty}\frac{s_f(n)}{n}=t,
$$
where $s_f(n)=\sum_{k=1}^na_f(k)$ denotes the summatory function.
One can pose many of the same questions and conjectures for these
sequences.

Another possible generalization is to consider the same recurrence formula
with different initial conditions. However, it turns out that this
offers no increase in generality, in the sense that if
$\{a_n'\}_{n\ge1}$ is any sequence satisfying
$$
a_n'=\pi(n)-\pi\bigl(\textstyle\sum_{k=1}^{n-1}a_k'\bigr)
\quad\text{for }n>n_0
$$
for some $n_0\ge0$, then $a_n'=a_n$ for all sufficiently large $n$.
(The same proof shows that taking $f(n)=n+c$ for some $c\in\ZZ$ in the
above, we have $a_f(n)=a_{n+c}$ for sufficiently large $n$.)
To see this, let $s_n'$ be the summatory sequence of $a_n'$, and set
$d_n=s_n'-s_n$. Swapping the roles of $a_n$ and $a_n'$ if necessary,
we may assume without loss of generality that $d_{n_0}\ge0$.
Then, as in the proof of Theorem~\ref{theorem:1}(iii), we find that
$0\le d_{n+1}\le d_n$. It follows that $d_n$ is eventually
constant, i.e.\ there exist $d\ge0$ and $n_1\ge n_0$ such that
$s_n'=s_n+d$ for all $n\ge n_1$. In turn this implies that $a_n'=a_n$
for all $n>n_1$.

At the same time, there are several possible avenues for further
research on $\{a_n\}$. We conclude with a few speculative suggestions.
\begin{enumerate}
\item
Assuming Cram\'er's conjecture, by Theorem~\ref{theorem:1}(ii) we have
$$
\#\{n\le x:a_n\ne0\}\ge\frac{\sum_{n\le x}a_n}{\max_{n\le x}a_n}
\gg\frac{x\sqrt{\log\log{x}}}{\log{x}}.
$$
This could be improved with some information on higher moment statistics
of $a_n$. For instance, can one give a non-trivial upper bound for
$\sum_{n\le x}a_n^2$?

\item
It is easy to see that the difference sequence $a_{n+1}-a_n$ is
almost always $0$, so $a_n$ has many long constant runs. (This idea
was used in the proof of Theorem~\ref{theorem:2}(iv).) Assuming
either Conjecture~\ref{conj:1}(A) or Dickson's conjecture, one
can see that for any $k\ge0$ there are arbitrarily long runs of
$n$ with $a_n=k$. Unconditionally, by Theorem~\ref{theorem:2}(i)
this holds for at least one $k\in\{0,1\}$, and from the proof of
Theorem~\ref{theorem:1}(iii) we get arbitrarily long runs on which $a_n$
is both constant and arbitrarily large. Can one give an unconditional
proof of long constant runs for a specific value of $k$?

\item
The previous question admits many generalizations. For instance,
assuming Dickson's conjecture, one can see that there are arbitrarily
long arithmetic progressions $n,n+d,\ldots,n+kd$ such that $a_{n+jd}=j$
for $j=0,\ldots,k$. Can this be proved unconditionally?
\end{enumerate}

\section*{Acknowledgements}
Altug Alkan would like to thank Robert Israel, Remy Sigrist and Giovanni
Resta for their valuable computational assistance regarding OEIS
contributions A335294 and A335337.

\bibliographystyle{amsplain}
\bibliography{A335294}
\end{document}